\documentclass[12pt]{amsart}
\usepackage{amscd,amsmath,amsthm,amssymb,verbatim,enumerate}
\usepackage{color}
\usepackage{amsfonts,latexsym,amsthm,amssymb,amsmath,amscd,euscript}
\usepackage{tikz-cd}
\usepackage[margin = 2cm]{geometry}
\usepackage[shortlabels]{enumitem}
\usepackage[utf8]{inputenc}
\usepackage{hyperref}
\usepackage{arydshln}
\usepackage{tikz}
\usepackage{float}
\usepackage{graphicx}
\usepackage[ruled, vlined]{algorithm2e}

\numberwithin{equation}{section}
\newcommand{\nc}{\newcommand}

\nc{\R}{\mathbb R}
\nc{\C}{\mathbb C}
\nc{\F}{\mathbb F}
\nc{\Q}{\mathbb Q}
\nc{\Z}{\mathbb Z}
\nc{\N}{\mathbb N}
\nc{\V}{\mathcal{V}}
\nc{\reg}{reg}

\newtheorem{theorem}{Theorem}[section]
\newtheorem{lemma}[theorem]{Lemma}

\theoremstyle{definition}
\newtheorem{definition}[theorem]{Definition}
\newtheorem{remark}[theorem]{Remark}
\newtheorem{example}[theorem]{Example}

\newtheorem{observation}[theorem]{Observation}

\begin{document}

\title{Down-left graphs and a connection to
toric ideals of graphs}

\author[Biermann]{Jennifer Biermann}
\address{Department of Mathematics and Computer Science, Hobart and William Smith Colleges \\
Geneva, NY 14456}
\email{biermann@hws.edu}

\author[Castellano]{Beth Anne Castellano}
\address{Department of Mathematics, Dartmouth College, Hanover, NH 03755}
\email{elizabeth.a.castellano.gr@dartmouth.edu}

\author[Manivel]{Marcella Manivel}
\address{Department of Mathematics, University of Minnesota, Minneapolis, MN 55455}
\email{
maniv013@umn.edu}

\author[Petruccelli]{Eden Petruccelli}
\address{Department of Mathematics and Statistics\\
McMaster University, Hamilton, ON, L8S 4L8}
\curraddr{31 Lower Horning Rd, Hamilton, ON L8S 3E9}
\email{Eden.petruccelli@gmail.com}

\author[Van Tuyl]{Adam Van Tuyl}
\address{Department of Mathematics and Statistics\\
McMaster University, Hamilton, ON, L8S 4L8}
\email{vantuyl@math.mcmaster.ca}

\date{\today}

\begin{abstract}
We introduce a family of graphs, which we call
down-left graphs, and study their 
combinatorial and algebraic properties.
We show that members of this 
family are well-covered, $C_5$-free, and vertex decomposable.   
By applying a result of H\`a-Woodroofe and 
Moradi--Khosh-Ahang, the (Castelnuovo-Mumford) regularity of the associated edge ideals is the induced  matching number of the graph.   As an application,
we  give a combinatorial
interpretation for the regularity of the toric
ideals of chordal bipartite graphs
that are $(K_{3,3} \setminus e)$-free.
\end{abstract}
\maketitle
%%%%%%%%%%%%%%%%%%%%%%%%%%%%%%%%%%%%%%%%%%%%%%%%%%%%%5

\section{Introduction}
Let $G = (V,E)$ be a finite simple graph with
vertex set $V = \{x_1,\ldots,x_n\}$ and edge set
$E = \{e_1,\ldots,e_d\}$ ($G$ is simple if it has
no loops or multiple edges).   Starting
with work of Villarreal \cite{V1}, there 
has been much interest in studying graphs 
through the lens of commutative algebra.  The bridge
between the two fields is via the edge ideal
construction, that is, given such a graph $G$,
the {\it edge ideal} is the square-free monomial
ideal $I(G) = \langle x_ix_j ~|~ \{x_i,x_j\} \in E \rangle$
in the polynomial ring $R = k[x_1,\ldots,x_n]$ with 
$k$ a field.  An active research program in 
combinatorial commutative algebra is to understand 
the interaction between the two subjects;
see, for example, \cite{HH,MRSW,VT,Vbook}.

The purpose of this paper is to introduce
and to study a family of graphs which we call
{\it down-left graphs}. This family was inspired by
a construction in
the paper of Biermann, O'Keefe, and Van Tuyl
\cite{BOVT} of graphs whose edge ideals arise as
the initial ideals of particular toric ideals.  While the formal definition (and generalization) will
be given in Definition \ref{defn:down-left}, we provide an
illustrative example that explains the name.
Let $1 \leq m \leq n$ be integers, and
let our vertex set be $V = \{x_{i,j} ~|~
1 \leq i \leq m, 1 \leq j \leq n \}$.  Arrange
our $mn$ vertices into an $m \times n$ grid, with vertex
$x_{i,j}$ in row $i$ and column $j$ (similar to matrix notation). We then attach vertex $x_{i,j}$ to 
all vertices
below and to the left of vertex $x_{i,j}$ in the grid.
The graph with $m =3$ and $n=4$, which
is denoted $G(3,4)$ in our notation, is shown in Figure \ref{downleftg34}.
\begin{figure}[h!]
\begin{tikzpicture}[scale=0.42]

\draw (0,0) -- (8,3);
\draw (0,0) -- (16,3);
\draw (0,0) -- (24,3);
\draw (0,0) -- (8,6);
\draw (0,0) to[out=35, in=190] (16,6);
\draw (8,0) to[out=35, in=190] (24,6);
\draw (0,0) -- (24,6);
\draw (8,0) -- (16,3);
\draw  (16,0) -- (24,3);
\draw (0,3) -- (8,6);
\draw (8,3) -- (16,6);
\draw (8,0) -- (24,3);
\draw (8,0) -- (16,6);
\draw (8,0) -- (24,6);
\draw  (16,3) -- (24,6);
\draw (16,0) -- (24,6);
\draw (0,3) -- (8,6);
\draw (0,3) -- (16,6);
\draw (0,3) -- (24,6);
\draw (8,3) -- (24,6);

\fill[fill=white,draw=black] (0,0) circle (.2) node[below]{$x_{3,1}$};
\fill[fill=white,draw=black] (0,3) circle (.2) node[left]{$x_{2,1}$};
\fill[fill=white,draw=black] (0,6) circle (.2) node[above]{$x_{1,1}$};
\fill[fill=white,draw=black] (8,0) circle (.2) node[below]{$x_{3,2}$};
\fill[fill=white,draw=black] (8,3) circle (.2) node[left]{$x_{2,2}$};
\fill[fill=white,draw=black] (8,6) circle (.2) node[above]{$x_{1,2}$};
\fill[fill=white,draw=black] (16,0) circle (.2) node[below]{$x_{3,3}$};
\fill[fill=white,draw=black] (16,3) circle (.2) node[right]{$x_{2,3}$};
\fill[fill=white,draw=black] (16,6) circle (.2) node[above]{$x_{1,3}$};
\fill[fill=white,draw=black] (24,0) circle (.2) node[below]{$x_{3,4}$};
\fill[fill=white,draw=black] (24,3) circle (.2) node[right]{$x_{2,4}$};
\fill[fill=white,draw=black] (24,6) circle (.2) node[above]{$x_{1,4}$};
\end{tikzpicture}
\caption{The down-left graph $G(3,4)$}\label{downleftg34}
\end{figure}
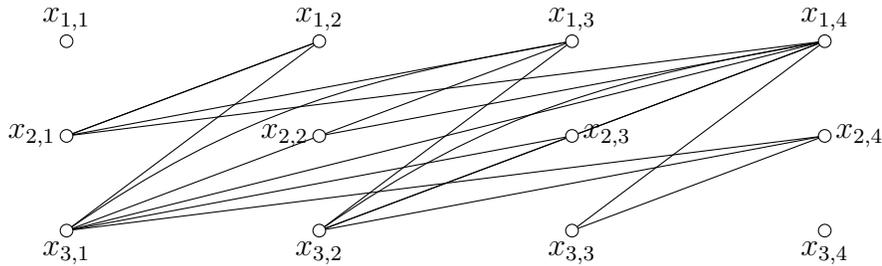
We will consider a more general family of graphs, denoted $G(m,n,\vec{a},\vec{b})$, where one is allowed to remove vertices of the grid from the lower-left and upper-right corners, as described by the vectors $\vec{a}$ and $\vec{b}$ (see Definition \ref{defn:generaldown-left}).

We prove that down-left graphs $G(m,n,\vec{a},\vec{b})$
satisfy a number of nice combinatorial proprieties,
which we now describe.  A graph is {\it well-covered}
if all of its minimal vertex covers have the 
same cardinality.  A graph is {\it $C_5$-free} if it contains
no induced subgraph isomorphic to a five cycle.  Finally,
a graph is {\it vertex decomposable} if it is the graph
of isolated vertices, or if it can be constructed
recursively from other vertex decomposable graphs (see
Definition \ref{defn:vd}).  
By combining our Theorems \ref{thm:wellcovered}, \ref{thm:downleft=>vd}, and \ref{5-cycle}  we prove:
\begin{theorem}\label{maintheorem1}
Let $G = G(m,n,\vec{a},\vec{b})$ be a down-left
graph.  Then the graph $G$ is well-covered, $C_5$-free, and 
vertex decomposable.
\end{theorem}
\noindent
Note that when $G$ is a vertex decomposable graph, 
the ring $R/I(G)$ is Cohen-Macaulay, e.g., see
\cite{DE,W}. While
we do not define the Cohen-Macaulay property here, 
Theorem \ref{maintheorem1} can be seen as part of the larger program of identifying
graphs with the property that $R/I(G)$ is
Cohen-Macaulay (for more see the survey \cite{MV}).

The results of Theorem \ref{maintheorem1} allow
us to apply a result of  H\`a-Woodroofe \cite{HW}
and Moradi--Khosh-Ahang 
\cite{MK} to compute the (Castelnuovo-Mumford)
regularity of the associated edge ideals of 
down-left graphs in terms of the
invariant ${\rm im}(G)$, the induced matching
number of $G$.

\begin{theorem}[{Theorem \ref{thm:reg-downleft}}]\label{maintheorem2}
If $G = G(m,n,\vec{a},\vec{b})$ is a down-left
graph, then  $${\rm reg}(R/I(G)) = {\rm im}(G).$$
\end{theorem}

As mentioned above, the original inspiration 
for down-left graphs came from the work
of Biermann, O'Keefe, and Van Tuyl \cite{BOVT} on  the regularity 
of toric ideals of chordal bipartite graphs,
that is, bipartite graphs with no induced
cycles of length six or larger. As an application
of Theorem \ref{maintheorem2}, we are able
to compute the regularity of toric ideas
of chordal bipartite graphs that are $(K_{3,3}\setminus
e)$-free (see Theorem \ref{thm:reg-chordalbipartite}). For this family of graphs,
this improves upon results of \cite{ADS,BOVT,HH1} which only give upper bounds for the regularity of 
toric ideals of (chordal) bipartite graphs
(see also \cite{B} which gives some exact formulas
in special cases).
Note that when $G$ is a bipartite graph, then 
computing ${\rm reg}(R/I_G)$, where $I_G$ is the 
toric ideal of $G$, is equivalent
to computing the $a$-invariant of $R/I_G$ (e.g., see \cite[Remark 2.12]{BVT}). While we 
omit this definition, the references
\cite{VV} and \cite[Section 11.5]{Vbook} 
can also be reinterpreted as results about the regularity of
toric ideals of bipartite graphs.

This paper uses the following
outline. In Section 2
we recall the relevant graph theory and commutative
algebra.  In Section 3 we derive our main
results about down-left graphs.  In our final section,
we apply our results to the toric ideals of chordal
bipartite graphs.

{\bf Acknowledgements.} This project started
as a Senior Research Project of 
Petruccelli at McMaster University under the
supervision of Van Tuyl.  The results of the project
were later expanded upon by Castellano and Manivel under the supervision
of Biermann as part of an REU supported by the National Science Foundation under grant no. DMS 1757616.
We thank Russ Woodroofe for his feedback
and Seyed Amin Seyed Fakhari for pointing out a 
gap in our original proof of Theorem \ref{thm:reg-downleft}.
Some of our results  were
inspired by computer calculations using Macaulay2 
\cite{Mt}.  Van Tuyl’s research is supported by NSERC Discovery Grant 2019-05412. 

%%%%%%%%%%%%%%%%%%%%%%%%%%%%%%%%%%%%%%%%%%%%
\section{Preliminaries}

In this section, we recall the necessary background from graph theory and commutative algebra.  Throughout
this paper, $k$ denotes a field
of characteristic zero.

\subsection{Graph theory background}
Let $G = (V,E)$ denote a finite
simple graph.  We will write $V(G)=V$, respectively $E(G)=E$, if we need to highlight
that we are referring to the vertices,
respectively edges, of $G$.

Given a subset $W \subseteq V$,
the {\it induced graph} of $G$ on $W$, denoted
$G_W$, is the graph $(W,E(G_W))$ where
$E(G_W) = \{e \in E(G) ~|~ e \subseteq W\}$.
We  say that $G$ is {\it $H$-free} if there
is no induced subgraph of $G$ isomorphic 
to $H$.  Given any edge $e \in E$ of $G$, the graph with the edge $e$ removed is denoted
$G \setminus e$.   Similarly, if $x \in V$ is 
a vertex, we write $G \setminus x$ to mean
the graph with the vertex $x$ and all edges
containing $x$ removed from $G$.
Given a vertex $x \in V$,
the {\it neighborhood} of $x$ is $N(x) = 
\{y ~|~ \{x,y\} \in E\}$.  The 
{\it closed neighborhood} of $x$ is 
$N[x] = N(x) \cup \{x\}$. 
If a graph $G$ has isolated vertices, then
we write $G^\circ$ for the graph formed by
removing all  the isolated vertices 
of $G$.

We require the following graphs.
The {\it $n$-cycle} is the graph $C_n = (\{x_1,\ldots,x_n\}, E(C_n))$ where 
$E(C_n) = \{\{x_i,x_{i+1}\} ~|~ i \in \{1,\ldots,n-1\}\} \cup \{x_n,x_1\}\}$.  A graph $G$
is a {\it bipartite graph} if the vertex
set can be partitioned as $V = V_1 \cup V_2$
such that every $e \in E$ has the property that
$e \cap V_i \neq \emptyset$ for $i=1,2$.
The {\it complete bipartite graph} 
$K_{m,n}$ has vertex set $\{x_1,\ldots,x_m,
y_1,\ldots,y_n\}$ and edge set
$\{\{x_i,y_j\} ~|~ 1 \leq i \leq m,~
1 \leq j \leq n\}$.
A {\it chordal bipartite graph} is a 
bipartite graph with no induced
cycles of length $\geq 6$.  
In Figure \ref{pictures} we have drawn
$C_5$ and $K_{3,3} \setminus e$ (where $e$ denotes any edge of $K_{3,3}$).
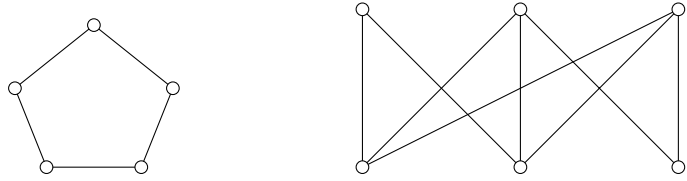
\begin{figure}
\begin{tikzpicture}[scale=0.42]

\draw (0,0) -- (3,0) -- (4,2.5) -- (1.5,4.5) -- (-1,2.5) -- (0,0);

\draw (10,0) -- (10,5) -- (15,0) -- (15,5) -- (20,0) -- (20,5) -- (15,0);
\draw (10,0) -- (15,5);
\draw (10,0) -- (20,5);

\fill[fill=white,draw=black] (0,0) circle (.2);
\fill[fill=white,draw=black] (3,0) circle (.2);
\fill[fill=white,draw=black] (4,2.5) circle (.2);
\fill[fill=white,draw=black] (1.5,4.5) circle (.2);
\fill[fill=white,draw=black] (-1,2.5) circle (.2);
\fill[fill=white,draw=black] (10,0) circle (.2);
\fill[fill=white,draw=black] (15,0) circle (.2);
\fill[fill=white,draw=black] (20,0) circle (.2);
\fill[fill=white,draw=black] (10,5) circle (.2);
\fill[fill=white,draw=black] (15,5) circle (.2);
\fill[fill=white,draw=black] (20,5) circle (.2);
\end{tikzpicture}
\caption{The graphs $C_5$ and $K_{3,3} \setminus e$}\label{pictures}
\end{figure}

A subset of vertices $W \subseteq V$ is 
a {\it vertex cover} if $e \cap W \neq \emptyset$ for all $e \in E$.
If $W$ is a vertex cover, the set complement
$V \setminus W$ is a called an {\it independent set}.
It has the property that  $e \not\subset (V\setminus W)$ for all $e \in E$.
A vertex cover $W$ is a {\it minimal vertex cover}
if no proper subset of $W$ is a vertex cover of $G$.
The complement of a minimal vertex cover is
called a {\it maximal independent set}.
A graph $G$ is {\it well-covered} if 
all of its minimal vertex covers (equivalently,
maximal independent sets) have the same
cardinality.  

A collection of edges $\{e_{i_1},\ldots,e_{i_k}\}$
is a {\it matching} in $G$ if the edges are pairwise
disjoint.  If $e_{i_j} = \{x_{i_j,1},x_{i_j,2}\}$,
then a matching is an {\it induced matching} of
$G$ if the induced graph on 
$\{x_{i_1,1},x_{i_1,2},\ldots,x_{i_k,1},x_{i_k,2}\}$
is the graph consisting only of the edges in
the matching.    The {\rm induced matching number} of
$G$, denoted ${\rm im}(G)$, is the number of edges
in the largest induced matching of $G$.

We shall also need the notion of a vertex
decomposable graph.

\begin{definition}\label{defn:vd}
A finite simple graph $G$ is a 
{\it vertex decomposable graph} if $G$ is 
well-covered and either $G$ consists only of
isolated vertices, or there exists a vertex
$x$ of $G$ such that $G \setminus x$ and 
$G \setminus N[x]$ are vertex decomposable.
The vertex $x$ is called
a {\it shedding vertex}.
\end{definition}

\begin{remark}
The notion of  vertex decomposability is originally due to Provan and Billera \cite{PB}, but
was defined for pure simplicial complexex, i.e., all
the facets of the simplicial complex have the same dimension.
The independence complex of $G$ is a simplicial
complex whose faces correspond to the 
independent sets of $G$. When we say a graph
$G$ is vertex decomposable, we are really 
saying that the independence complex of $G$
is a vertex decomposable simplicial complex with
respect to the definition of Provan and Billera.
Instead of defining the independence complex, we have
defined vertex decomposability directly in terms of
the graph.  The justification for this alternate 
formulation can be found in Dochtermann and Engstr\"om 
\cite{DE} and Woodroofe \cite{W}.  
Note that in \cite{DE,W}, the 
authors use the non-pure definition of 
vertex decomposability of Bj\"orners and Wach \cite{BW};  when the simplicial complexes
are pure, the definition reverts to that of 
\cite{PB}.  Pure independence complexes correspond
to well-covered graphs.
\end{remark}

To check if a graph
is vertex decomposable, it is enough to show that
all the connected components of $G$ have this property.

\begin{lemma}[{\cite[Lemma 20]{W}}]\label{components}
    Suppose that the graph 
    $G = G_1 \sqcup G_2$ is the disjoint union of the
    graphs $G_1$ and $G_2$.  If $G_1$ and $G_2$ are vertex
    decomposable, then $G$ is vertex decomposable.
\end{lemma}

\subsection{Commutative algebra background}
We now turn to the needed results from
commutative algebra.
If $I$ is a homogeneous ideal of the polynomial ring
$R = k[x_1,\ldots,x_n]$, then 
there is a {\it minimal graded free resolution} associated with $R/I$.  That is,
there exists a long exact sequence of the form 
\[0 \rightarrow \bigoplus_{j \in \mathbb{N}} R(-j)^{\beta_{p,j}(R/I)} 
 \rightarrow \cdots
 \rightarrow \bigoplus_{j \in \mathbb{N}} R(-j)^{\beta_{1,j}(R/I)} 
\rightarrow R \rightarrow R/I \rightarrow 0\]
that is minimal with respect to the exponents on the free modules, where $R(-j)$ is the graded $R$-module obtained by shifting the 
degrees of $R$ by $j$ and where $p \leq n$.  The number
$\beta_{i,j}(R/I)$ is the $(i,j)$th {\it graded Betti number} of $R/I$.  For a detailed treatment see \cite{P}.   

The edge ideal $I(G)$ of the graph $G$, as defined in 
the introduction, and the toric ideal $I_G$ of a graph 
$G$, to be defined in Section 4, are both 
homogeneous ideals.  Consequently,
they have a minimal graded  free resolution, and one
is interested in how the graph theoretical
invariants are encoded into the graded
Betti numbers.  Of particular interest
is the following invariant, which is
a rough measure of the complexity of 
$R/I$.

\begin{definition}
For any homogeneous ideal $I$ of $R$, 
the {\it Castelnuovo-Mumford regularity} (or regularity) of $R/I$ is 
\[{\rm reg}(R/I) = \max\{j-i ~|~ \beta_{i,j}(R/I) \neq 0\}. \]
\end{definition}

When computing regularity of edge ideals, we can
restrict to connected components.

\begin{theorem}[{\cite[Proposition 3.4]{MV}}]\label{regconnected}
Suppose that the graph $G = G_1 \sqcup G_2$ is
the disjoint union of the graphs $G_1$ and $G_2$.
Then 
${\rm reg}(R/I(G)) = {\rm reg}(R/I(G_1)) +
{\rm reg}(R/I(G_2)).$
\end{theorem}

The following two results
link some of the concepts introduced in this section.

\begin{theorem}[{\cite[Lemma 2.2]{K}}]\label{thm:katzmanbound}
    Let $G$ be any finite simple graph.  Then
    ${\rm im}(G) \leq {\rm reg}(R/I(G)).$
\end{theorem}

The following result is a special case of 
{\cite[Theorem 4.2]{HW} and  
\cite[Corollary 2.11]{MK}}, applied
to pure independence complexes of graphs
that are vertex decomposable.

\begin{theorem}\label{thm.recursivereg}
    Let $G$ be a vertex decomposable graph. If $x$ is a shedding vertex
    of $G$, then
    $${\rm reg}(R/I(G)) = \max\{{\rm reg}(R/I(G\setminus x)), {\rm reg}(R/I(G\setminus N[x])+1\}.$$
\end{theorem}

%The following result of Khosh-Ahang and Moradi
%links some of the concepts introduced in
%this section.  In particular, the regularity
%of the edge ideals of 
%a class of vertex decomposable graphs  can be related to %the induced
%matching number.

%\begin{theorem}[{\cite[Theorem 2.4]{KM}}]\label{thm:kam}
%Let $G$ be a $C_5$-free vertex decomposable graph.  Then %${\rm reg}(R/I(G)) = {\rm im}(G)$.
%\end{theorem}

In Section 4, we also require
the following result of Conca and Varbaro.  
If $>$ is an monomial order on 
$R = k[x_1,\ldots,x_n]$, then 
the {\it initial ideal} of an ideal $I$ is the 
monomial ideal ${\rm in}_>(I) =
\langle LM_<(f) ~|~ f\in I \rangle$, where
$LM_<(f)$ denotes the leading monomial
of $f$.  

\begin{theorem}[{\cite[Corollary 2.6]{CV}}]\label{regSqfreeInitialIdeal}  Let $I$ be a homogeneous ideal in $R$.
Suppose there is 
a monomial order $>$ such that
$J = {\rm in}_>(I)$ is a square-free
monomial ideal. Then ${\rm reg}(R/I) = {\rm reg}(R/J)$.
\end{theorem}

%%%%%%%%%%%%%%%%%%%%%%%%%%%%%%%%%%%%%%%%%%%%%%%
\section{Down-left graphs and their properties} \label{GenThm3.2}

We introduce our main object of study,  down-left graphs, and study some 
of their properties.  In Section 4, the edge ideals of these graphs arise as the initial ideals of toric ideals in certain cases.  The properties established here will allow us to compute the regularity of these edge ideals and hence, of the more complicated toric ideals.

All down-left graphs will be constructed
as induced subgraphs of the following family.

\begin{definition}\label{defn:down-left}
The graph $G(m,n)$ is  the graph on vertex set $V = \{x_{i,j} \mid 1 \leq i \leq m, 1, \leq j\leq n\}$  with edge set $E = \{\{x_{i,j}, x_{k, \ell}\} \mid i <k, j>\ell\}$.
\end{definition}

Figure \ref{downleftg34} shows the graph $G(3,4)$.   We now define down-left graphs,
which are formed by (possibly) removing vertices
in the bottom-left corner or upper-right corner
of $G(m,n)$.

\begin{definition}\label{defn:generaldown-left}
Let $\vec{a} = (a_1, a_2, \dots, a_m)$ and $\vec{b} = (b_1, b_2, \dots, b_m)$ be vectors in $\Z_{\geq 0}^m$ such that $0 =a_1 \leq a_2 \leq a_3 \leq \dots \leq a_m<n$, $1<b_1 \leq b_2 \leq \dots \leq b_{m-1}\leq b_m = n+1$ and $a_i <b_i$ for all $2 \leq i \leq m-1$.    These sequences correspond to sets of vertices in the graph $G(m,n)$ as follows
\begin{align*}
W &= \{x_{2,1}, x_{2,2}, \dots, x_{2, a_2}, x_{3,1}, x_{3,2}, \dots, x_{3, a_3}, \dots, x_{m, 1}, x_{m,2}, \dots, x_{m, a_m}\}\\
Z &=  \{x_{1,b_1}, x_{1, b_1+1}, \dots, x_{1, n}, x_{2,b_2}, x_{2, b_2+1} \dots, x_{2, n}, \dots, x_{m-1, b_{m-1}}, x_{m-1, b_{m-1}+1}  \dots, x_{m-1, n}\}.
\end{align*}
Note that if $a_i =0$, respectively $b_j =n+1$,
the vertex $x_{i,a_i}$, respectively $x_{j,b_j}$
is not included in $W$, respectively $Z$.
The graph $G(m,n,\vec{a}, \vec{b})$ is the induced subgraph of $G(m,n)$ on the vertex set $V(G(m,n))\setminus (W\cup Z)$.  We call this class of graphs  \emph{down-left graphs}.
Note that  $G(m,n) = G(m,n,\vec{0},(n+1,\ldots,n+1))$
since $Z = W = \emptyset$.
\end{definition}

\begin{observation}\label{observation}
If $a_i + 1 =b_i$, then $Z \cup W$ contains
all the vertices $$\{x_{i,1},\ldots,x_{i,a_i},x_{i,a_i+1},
\ldots,x_{i,n}\}.$$  That is, none of the vertices
in the $i$-th row appear. Consequently,
$$G(m,n,\vec{a},\vec{b}) = G(m-1,n,(a_1,\ldots,\hat{a}_i,
\ldots,a_n),(b_1,\ldots,\hat{b}_i,\ldots,b_n)).$$
By removing any such ``empty'' rows, we can assume
that $a_i +1 < b_i$ for all $i$.

Similarly, if $b_{i-1} \leq a_i+1$, then $G$
is the disjoint union of two down-left graphs, namely
$$G(i-1,b_{i-1}-1,(a_1,\ldots,a_{i-1}),(b_1,\ldots,b_{i-1}))$$
and 
$$G(m-i+1,n-a_i,(a_i-a_i,\ldots,a_m-a_i),(b_i-a_i,
\ldots,b_n-a_i))$$
Consequently, it is enough to only consider
the down-left graphs with $b_{i-1} > a_i+1$.
\end{observation}

Informally, the graph $G(m,n, \vec{a}, \vec{b})$ is 
the graph $G(n,m)$ with  staircases of vertices (and all adjacent edges) removed from the bottom-left and top-right portions of the graph.  We include an example
to illustrate this family.

\begin{example} The graph
$G(5,6,\vec{a},\vec{b})$ with $\vec{a} = (0,0,1,2,2)$ and
$\vec{b}=(5,5,6,6,7)$ is given
in Figure \ref{g(5,6,a,b)}. Following Definition 
\ref{defn:generaldown-left}, the
sets $W$ and $Z$ are 
$$
W = \{x_{2,0},x_{3,1},x_{4,1},x_{4,2},x_{5,1},x_{5,2}\} 
~\mbox{and}~
Z = \{x_{1,5},x_{1,6},x_{2,5},
x_{2,6},x_{3,6},x_{4,6},x_{5,7}\}.
$$
We omit the vertex $x_{2,0}$ from
$W$ and $x_{5,7}$ from $Z$.
We then consider the induced
subgraph of $G = G(5,6)$ on the 
vertex set $V(G) \setminus (W \cup Z)$.
Note that we have
supressed the labelling of the vertices.  However,
we are assuming that vertex $x_{i,j}$ appears in the
$i$th row and $j$th column of the grid (using 
the convention of matrix notation).
For reference, all of the vertices
of $G(5,6)$ are included;  the graph $G(5,6,\vec{a},\vec{b})$ 
does not include the vertices
that are also in $W\cup Z$, which are denoted by 
the solid black vertices.  The graph
$G(5,6,\vec{a},\vec{b})^\circ$ is the graph
we obtain by removing all the isolated vertices
in Figure \ref{g(5,6,a,b)}.
\begin{figure}[h!]
\begin{tikzpicture}[scale=0.27]

\draw (0,9) -- (8,12);
\draw (0,9) -- (16,12);
\draw (0,9) -- (24,12);

\draw (8,9) -- (16,12);
\draw (8,9) -- (24,12);

\draw (16,9) -- (24,12);

\draw (8,6) -- (16,12);
\draw (8,6) -- (16,9);
\draw (8,6) to[out =35, in =190] (24,12);
\draw (8,6) -- (24,9);

\draw (16,6) -- (24,9);
\draw (16,6) -- (24,12);

\draw (16,3) -- (24,6);
\draw (16,3) -- (24,9);
\draw (16,3) -- (24,12);
\draw (16,3) -- (30,6);

\draw (24,3) -- (30,6);

\draw (16,0) -- (24,3);
\draw (16,0) -- (24,6);
\draw (16,0) -- (24,9);
\draw (16,0) -- (24,12);
\draw (16,0) -- (30,3);
\draw (16,0) to[out=35, in=190] (30,6);

\draw (24,0) -- (30,3);
\draw (24,0) -- (30,6);

%\draw (0,0) -- (16,3);
%\draw (0,0) -- (24,3);
%\draw (0,0) -- (8,6);
%\draw (0,0) to[out=35, in=190] (16,6);
%\draw (8,0) to[out=35, in=190] (24,6);
%\draw (0,0) -- (24,6);

%\draw (8,0) -- (16,3);
%\draw  (16,0) -- (24,3);

%\draw (0,3) -- (8,6);
%\draw (8,3) -- (16,6);
%\draw (8,0) -- (24,3);
%\draw (8,0) -- (16,6);
%\draw (8,0) -- (24,6);
%\draw  (16,3) -- (24,6);
%\draw (16,0) -- (24,6);
%\draw (0,3) -- (8,6);
%%\draw (0,3) -- (16,6);
%\draw (0,3) -- (24,6);
%\draw (8,3) -- (24,6);
\fill[fill=black,draw=black] (0,0) circle (.2);
\fill[fill=black,draw=black] (0,3) circle (.2);
\fill[fill=black,draw=black] (0,6) circle (.2);
\fill[fill=white,draw=black] (0,9) circle (.2);
\fill[fill=white,draw=black] (0,12) circle (.2);
\fill[fill=black,draw=black] (8,0) circle (.2);
\fill[fill=black,draw=black] (8,3) circle (.2);
\fill[fill=white,draw=black] (8,6) circle (.2);
\fill[fill=white,draw=black] (8,9) circle (.2);
\fill[fill=white,draw=black] (8,12) circle (.2);
\fill[fill=white,draw=black] (16,0) circle (.2);
\fill[fill=white,draw=black] (16,3) circle (.2);
\fill[fill=white,draw=black] (16,6) circle (.2);
\fill[fill=white,draw=black] (16,9) circle (.2);
\fill[fill=white,draw=black] (16,12) circle (.2);
\fill[fill=white,draw=black] (24,0) circle (.2);
\fill[fill=white,draw=black] (24,3) circle (.2);
\fill[fill=white,draw=black] (24,6) circle (.2);
\fill[fill=white,draw=black] (24,9) circle (.2);
\fill[fill=white,draw=black] (24,12) circle (.2);
\fill[fill=white,draw=black] (30,0) circle (.2);
\fill[fill=white,draw=black] (30,3) circle (.2);
\fill[fill=white,draw=black] (30,6) circle (.2);
\fill[fill=black,draw=black] (30,9) circle (.2);
\fill[fill=black,draw=black] (30,12) circle (.2);
\fill[fill=white,draw=black] (36,0) circle (.2);
\fill[fill=black,draw=black] (36,3) circle (.2);
\fill[fill=black,draw=black] (36,6) circle (.2);
\fill[fill=black,draw=black] (36,9) circle (.2);
\fill[fill=black,draw=black] (36,12) circle (.2);
\end{tikzpicture}
\caption{The graph $G(5,6,\vec{a},\vec{b})$ with 
$\vec{a} = (0,0,1,2,2)$ and $\vec{b} = (5,5,6,6,7)$. The black vertices belong to $G(5,6)$, but not
$G(5,6,\vec{a},\vec{b})$. The
graph $G(5,6,\vec{a},\vec{b})^\circ$ is obtained
by removing all the isolated vertices.}\label{g(5,6,a,b)}
\end{figure}
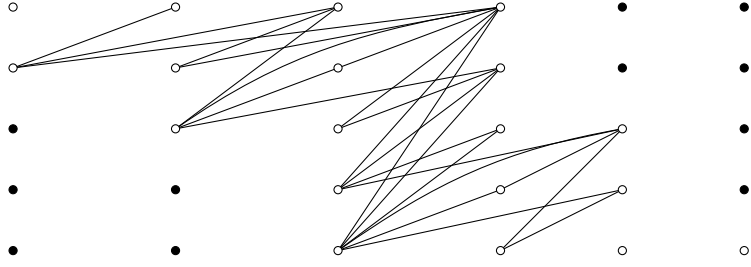

Informally, the vector
$\vec{a}$ is describes how
many vertices we remove 
from each row starting
from the the left, i.e., 0 vertices
from row 1, 0 vertices from row 2, 1 vertex from row 3, 2 vertices
from row 4,and 2 vertices from row 5.  The vector $\vec{b}$ 
describes where to remove vertices
from the righthand side, i.e., 
starting with the fifth vertex
of row 1, remove all the remaining vertices from that row, in the second row, remove the fifth vertex
and all remaining vertices in that row, and so on.
\end{example}

\begin{remark}
We have defined a down-left graph in such a way that it always contains isolated vertices.  Indeed, the vertices $x_{1,1}$ and $x_{m,n}$ are always isolated and there may be more isolated vertices, as
seen in the previous example.
\end{remark}

We now prove that all down-left
graphs are well-covered.  We start with
some lemmas.

\begin{lemma}\label{lem:alongsamerow}
Let $G = G(m,n,\vec{a},\vec{b})$.
Suppose that $W$ is a maximal
independent set of $G$.  If
$x_{i,j}, x_{i,k} \in W$, then
$x_{i,\ell} \in W$ for all 
$j < \ell < k$.  Similarly,
if $x_{i,j},x_{k,j} \in W$, then
$x_{\ell,j} \in W$ for all
$i < \ell < k$.
\end{lemma}

\begin{proof}
Suppose $x_{i,j}, x_{i,k} \in W$, and 
suppose that there is an integer $\ell$ such
tht $j < \ell < k$.  We wish to show
that $x_{i,\ell}$ is also in $W$.
Suppose
$x_{i,\ell} \not\in W$.  
Because $W$ is a maximal independent set, this would
mean that $x_{i,\ell}$ is adjacent
to some vertex, say $z$, in $W$.  

Now $x_{i,\ell}$ is adjacent to all the vertices
below row $i$ and to the left of 
column $\ell$, and to all the vertices
above row $i$ and to the right of column
$\ell$. If the vertex $z$ is
below row $i$ and to the left of
column $\ell < k$, then it is also
adjacent to $x_{i,k}$.  Since 
$x_{i,k} \in W$, $z$ cannot have this
property.  On the other hand, if
$z$ is above row $i$ and to the right
of column $ j  < \ell$, then $z$
is also adjacent to $x_{i,j}$.  But
this would mean $z$ and $x_{i,j}$ are adjacent in an independent set, which is
a contradiction.  So, no such $z$ exists.
Thus $x_{i,\ell} \in W$, as desired.

The second statement is proved similarly.
\end{proof}
\begin{lemma}\label{eitheror}
Let $G = G(m,n,\vec{a},\vec{b})$ with
$a_i +1 < b_{i-1}$ for all $i$.
Suppose that $W$ is a maximal
independent set of $G$.  If 
$x_{i,j} \in W$ with $(i,j) \neq (m,n)$, then exactly one
of $x_{i+1,j}$ or $x_{i,j+1}$ is also in
$W$.
\end{lemma}

\begin{proof}
Let $W$ be a maximal independent
set of $G$ with $x_{i,j} \in W$.  

We first claim that at least one of
$x_{i+1,j}$ and $x_{i,j+1}$ also belong to
$V(G)$. (We need to first check this claim since these 
vertices may have been removed when constructing
$G$).  Note that if $i=m$, then since $x_{m,n} \in V(G)$,
we must have $x_{i,j},x_{i,j+1},\ldots,x_{i,n}=x_{m,n} \in V(G)$,
by the construction of $G$. Since $(i,j) \neq (m,n)$, this means that $x_{i,j+1} \in V(G)$.  Now suppose
$i < m$ and suppose neither $x_{i+1,j}$
nor $x_{i,j+1}$ are in $V(G)$.
By Observation \ref{observation}, because $a_{i+1} +1 < b_i \leq b_{i+1}$,
there are some vertices in the $(i+1)$-th row.  
If $x_{i+1,j}$ is not in $V(G)$ either $j \leq a_{i+1}$ or $j \geq b_{i+1}$. However we know that since $x_{i,j} \in V(G)$, $j < b_i \leq b_{i+1}$ and thus we must have $j\leq a_{i+1}$.  On the other hand, if $x_{i,j+1}$ is not
in $V(G)$, either $j+1 \leq a_i$ or $b_i \leq j+1$. Again since $x_{i,j} \in V(G)$ we know $j>a_i$ and thus $j+1>a_i$.
But then $b_i \leq j+1 \leq a_{i+1}+1 < b_i$, which
is a contradiction.  So, at least one
of $x_{i+1,j}$ and $x_{i,j+1}$ is in $V(G)$.

Note that if
both of the vertices 
$x_{i+1,j}$ and $x_{i,j+1}$ are
vertices in $G$, then they are 
adjacent, since $x_{i+1,j}$ is both
below and to the left of $x_{i,j+1}$.  
So, at most one of these two vertices can
belong to $W$.

So, suppose that neither of
these two vertices belong to $W$.
Since $W$ is a maximal independent
set, if we add either vertex to $W$,
the set is no longer independent, i.e.,
the vertex that we add is adjacent to
a vertex in $W$.  Suppose we add
$x_{i+1,j}$, and it is adjacent to $z \in W$.  Now $z$ cannot be below and to
the left of $x_{i+1,j}$.  If it was,
then it is also down and left of $x_{i,j}$, and so it would be adjacent
to $z$, contradicting our assumption that $W$ is an independent set.
So, $z$ must be above and to the 
right of $x_{i+1,j}$, i.e., 
$z = x_{r,s}$ with $r < i+1$ and $s >j$.
If $r < i$, then $z$ is above and to 
the right of $x_{i,j}$, so the two
vertices are adjacent, contradicting
our assumption that $W$ is independent.
So $z = x_{i,s}$ with $s >j$.  But then
by Lemma \ref{lem:alongsamerow}, we must
also have $x_{i,j+1} \in W$, which 
contradicts our assumption that
$W$ does not contain either of these
two vertices.

We derive a similar conclusion
if $W$ is a maximal independent set and if we 
add $x_{i,j+1}$.
\end{proof}

\begin{theorem}\label{thm:wellcovered}
Every down-left graph 
$G = G(m,n,\vec{a},\vec{b})$ is a
well-covered graph. 
\end{theorem}

\begin{proof}
A graph $G$ is well-covered if and only if 
all its connected components are well-covered.  Thus, 
by Observation \ref{observation}, it is enough
to show that $G$ is well-covered when
$a_i+1 < b_{i-1} \leq b_i$ for all $i$.
Let $W$ be a maximal independent
set of $G$.  Since the vertex
$x_{1,1}$ is isolated in $G$,
$x_{1,1}$ belongs to $W$.  By 
Lemma \ref{eitheror}, either 
$x_{2,1}$ or $x_{1,2}$ is in $W$.
If $x_{2,1}$ is in $W$, then either
$x_{3,1}$ or $x_{2,2}$ is in $W$;  similarly, if $x_{1,2} \in W$, we have
exactly one of $x_{2,2}$ or $x_{1,3}$
in $W$.  The process repeats until the 
vertex $x_{m,n}$ is shown to be part
of $W$.  

In other words, each maximal
independent set $W$ describes a lattice
path from the vertex $x_{1,1}$ to
$x_{m,n}$, moving either right one 
column or
down one row at a time.  No path will
stop before getting to $x_{m,n}$ as shown by
Lemma \ref{eitheror}.
The length of any
such path is $m+n-1$.  So, all
the maximal independent sets have the
same size, i.e., the graph is 
well-covered.
\end{proof}

Next we will show that down-left graphs are vertex decomposable.

\begin{theorem}\label{thm:downleft=>vd}
Every down-left graph $G(m,n, \Vec{a},\Vec{b})$ is 
vertex decomposable.
\end{theorem}

\begin{proof}
Let $G = G(m,n,\vec{a}, \vec{b})$ be a down-left graph.
By repeatedly applying Observation \ref{observation}, the graph
$G$ can be written as the disjoing union of 
smaller down-left graphs, all with the property that 
that $a_i+1 < b_i$ for all $i$ and $a_i+1 < b_{i-1}$ for
$i=2,\ldots,m$.  By
Lemma \ref{components}, it thus suffices to show that
a down-left graph $G$ is vertex decomposable
under these extra assumptions on $\vec{a}$ and $\vec{b}$.
By Theorem \ref{thm:wellcovered}, the graph $G$ is 
well-covered.  

We proceed by induction on the number of vertices $d$ in $G$.  The base case $d=1$, when $G$ has a single vertex, is vertex decomposable by definition.

Suppose that $G$ has $d \geq 2$ vertices and that any down-left graph with fewer than $d$ vertices is a vertex decomposable graph.  Let $i$ be the largest index such that $a_i=0$ (recall that by definition, $a_1 = 0$).  Let $x = x_{i,1}$.

If $i = 1$, then $x_{1,1}$ is a vertex of $G$, but
$x_{2,1},\ldots,x_{m,1}$ are not vertices of $G$ (since
$1 \leq a_j$ for $2 \leq j \leq m$.). In this case $x$ has no neighbors and so $G \setminus x$ and $G \setminus N[x]$ are both the graph $G(m, n-1, (0, a_2-1, \ldots, a_n-1),(b_1-1, b_2-1, \dots, b_n-1))$.  Informally, in this case we
are removing all the vertices in the first column.
By induction, this graph is
vertex decomposable, so $G$ is vertex decomposable.

Now assume that $i>1$.  First consider $G \setminus x$.   The graph $G \setminus x$ is of the form $G(m,n,\Vec{a'}, \Vec{b})$ where $\Vec{a'}$ only differs from $\Vec{a}$ in that $a'_i=1$.  Note that the way we chose $x$ ensures that $a_i < a_{i+1}$, so $a'_i \leq a_{i+1}=a'_{i+1}$ and thus $\vec{a'}$ is still a non-decreasing sequence.  Further, since $i>1$, $a_1' = a_1 =0$ and since $x_{i,1}$ is a vertex of $G$, $b_i \geq 2$ which means $a_i' \leq b_i$.  This ensures that $G \setminus x = G(m,n, \vec{a'}, \vec{b})$ is a down-left graph on $d-1$ vertices. (Note that it is possible the new
down-left graph must be first broken down into smaller 
down-left graphs using Observation \ref{observation} to
ensure the correct conditions on $\vec{a'}$ and $\vec{b}$;  however, by induction, the smaller subgraphs are vertex
decomposable, so by Lemma \ref{components}, $G(m,n,
\vec{a'},\vec{b})$ is vertex decomposable).
Therefore, by induction $G \setminus x$ is vertex decomposable.

Next consider $G \setminus N[x]$.  Since $x_{i,1}$ is adjacent
to all vertices to the right and above $x_{i,1}$, removing
$N[x]$ is equivalent to changing $a_i=0$ in $\vec{a}$
to $a_i=1$ and $b_1,\ldots,b_{i-1}$ in $\vec{b}$ 
to $b_1=\cdots =b_{i-1}=2$ in $\vec{b}$.  Let $\vec{a'}$ and
$\vec{b'}$ denote these new new vectors.
Note that $\vec{a'}$ is still non-decreasing with $a_1 =0$
and $\vec{b'}$ is still non-decreasing with $b_m = n+1$.
Thus $G \setminus N[x]$ is the down-left graph
$G(n,m,\Vec{a'},\Vec{b'})$.  By induction, this 
smaller down-left graph is vertex decomposable (again,
one might first need to use Observation \ref{observation}
to decompose this down-left graph into smaller down-left
graphs). 

Consequently, $G$ is vertex decomposable, as desired.
\end{proof}

Down-left graphs also  do not contain the 5-cycle as an induced subgraph.

\begin{theorem}\label{5-cycle}
Every down-left graph $G(m,n,\vec{a},\vec{b})$ is $C_5$-free.
\end{theorem}

\begin{proof}
Let $G = G(n,m,\vec{a},\vec{b})$.
Assume that there is an induced subgraph of $G$ that is isomorphic to a 5-cycle. Then there is a set of five vertices $H=\{v_1,v_2,v_3,v_4,v_5\} \subset V(G)$ where each vertex $v_i$ is adjacent to exactly two others in $H$. We will denote the column of a vertex $v_i$ as $n(v_i)$ and its row as $m(v_i)$. Without loss of generality, let $v_2$ be the rightmost vertex of $H$ and the upper one in the case more than one are in the same column.

This vertex $v_2$ is adjacent to two other vertices,
say $v_1$ and $v_3$, that must be down and to the left of $v_2$ but not adjacent to each other (i.e., neither $v_1$ nor $v_3$ is down and to the left of the other).  
 Without loss of generality let $v_1$ be the further left and/or upper of the two. There are three
possible positions for $v_3$: in the same column
as $v_1$, in the same row as $v_1$, or 
below and to the right of $v_1$, as shown in
the Figure \ref{fig:C5-free}.

\begin{figure}[h!]
    \centering
    \includegraphics[scale=0.65]{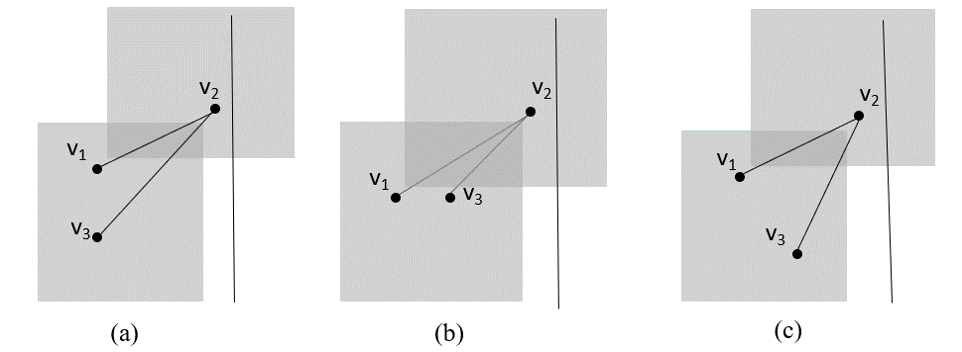}
    \caption{Three possible ways to place three vertices without an induced triangle. Grey areas represent places that are adjacent to either $v_1$ or $v_2$.}
    \label{fig:C5-free}
\end{figure}

If $v_1$ and $v_3$ are in the same column (Figure \ref{fig:C5-free} (a)), then every vertex that is adjacent to $v_1$ is either up and to the right of $v_1$ and thus also adjacent to $v_3$, or down and to the left of $v_1$ and thus also adjacent to $v_2$.  This contradicts the assumption that $\{v_1, v_2, v_3, v_4, v_5\}$ is an induced 5-cycle.

Similarly, if $v_1$ and $v_3$ are in the same row (Figure \ref{fig:C5-free} (b)), then every vertex which is adjacent to $v_3$ is adjacent to either $v_1$ or $v_2$, which is again a contradiction.

Now consider the case where $v_1$ and $v_3$ do not share a row or a column (Figure \ref{fig:C5-free} (c)).  Suppose that $v_4$ is adjacent to $v_3$ and $v_5$ is adjacent to $v_1$.  Since $v_4$ is adjacent to $v_3$ but not adjacent to $v_2$, $v_4$ must be up and to the right of $v_3$, but not down and to the left of $v_2$.  Since $v_2$ was assumed to be the vertex of the 5-cycle that is furthest right and uppermost, we must have $v_4$ in the same column as and below $v_2$.  Thus $n(v_4) = n(v_2)$ and $m(v_4)>m(v_2)$.  Similarly, since $v_5$ is adjacent to $v_1$ but not adjacent to $v_2$ we must have $m(v_5) = m(v_2)$ and $n(v_5)<n(v_2)$.  Together these imply that $n(v_4) >n(v_5)$ while $m(v_5)<m(v_4)$ which contradicts the fact that $v_4$ and $v_5$ are adjacent.  
\end{proof}

By combining Theorem \ref{thm:wellcovered}
with Theorem \ref{thm.recursivereg}, we can 
now compute the regularity of all down-left
graphs in terms of an invariant of the graph.

\begin{theorem}\label{thm:reg-downleft}
If $G = G(n,m,\vec{a},\vec{b})$ is a down-left graph, then
${\rm reg}(R/I(G)) = {\rm im}(G).$
\end{theorem}

\begin{proof}
    By Theorem \ref{thm:katzmanbound}
    we have ${\rm im}(G) \leq {\rm reg}(R/I(G))$.  
    We now show ${\rm im}(G)$ is also an upper bound.
    
    Like in the proof of Theorem \ref{thm:downleft=>vd}, we can
    repeatedly apply Observation \ref{observation}
    to assume that $G$ is the disjoint union
    of smaller down-left graphs with
    $a_i+1 < b_i$ for all $i$ and $a_i+1 < b_{i-1}$ for all $2 \leq i \leq m$.
    By Lemma \ref{regconnected}, it suffices
    to prove the statement under these
    hypotheses on $\vec{a}$ and $\vec{b}$.  

    We do induction on $d$, the number of 
    vertices of $G$.  If $d=1$, then $G$ is 
    a single vertex. In this case $I(G) = \langle
    0 \rangle$, so ${\rm reg}(R/I(G)) = 0 = {\rm im}(G)$.

    Suppose that the graph has $d \geq 2$ vertices
    and the statement holds for all down-left graphs
    with fewer than $d$ vertices.  Let $i$ be the largest
    index such that $a_i=0$  (since $a_1=0$, there
    is at least one such index).  Let $x= x_{i,1}$.

    As in the proof of Theorem \ref{thm:downleft=>vd},
    if $i=1$, then $x_{i,1}$ is an isolated vertex and there are no vertices in the first column below $x_{i,1}$.  Thus $I(G\setminus x)$ is also a down-left graph and $I(G) = I(G\setminus x)$.  By
    induction we have ${\rm reg}(R/I(G)) = 
    {\rm im}(G\setminus x) = {\rm im}(G)$.

    If $i >1$, then by assumption 
    $b_{i-1} > a_i+1 =1$.  If
    $b_{i-1}=2$, then $b_1 = \cdots = b_{i-1}=2$.  
    But this means that all the vertices above
    and to the right of $x_{i,1}$ are not in 
    $G$, and $x_{1,1},x_{2,1},\ldots,x_{i-1,1}$ are isolated
    vertices.  Removing these isolated vertices
    from $G$ gives us a new down-left graph $H$
    with $I(H) = I(G)$. By induction,
    ${\rm reg}(R/I(G)) = {\rm reg}(R/I(H))
    = {\rm im}(H) = {\rm im}(G)$.

    So, suppose that $b_{i-1} > 2$.  In 
    Theorem \ref{thm:downleft=>vd} we proved that $x_{i,1}$
    is a shedding vertex, where $G\setminus x$ and
    $G \setminus N[x]$ are also down-left graphs
    on fewer vertices. Thus, by Theorem
    \ref{thm.recursivereg} and induction we have
    \begin{eqnarray*}
    {\rm reg}(R/I(G)) &=& \max\{{\rm reg}(R/I(G\setminus
    x), {\rm reg}(R/I(G\setminus N[x])+1\} \\
    &=& \max\{{\rm im}(G\setminus x), {\rm im}(G \setminus
    N[x])+1\}.
    \end{eqnarray*}
    It is clear that ${\rm im}(G \setminus x) \leq 
    {\rm im}(G)$, so it suffices to show that 
    ${\rm im}(G \setminus N[x])+1 \leq {\rm im}(G)$.
    Since $b_{i-1} > 2$, we have $x_{i-1,2}$ is 
    a vertex of $G$, and moreover, $\{x,x_{i-1,2}\}$
    is an edge of $G$ since it is above and to the right
    of $x$.  Furthermore, note that $N[x_{i-1,2}] \subseteq N[x]$ since any vertex above and
    to the right of $x_{i-1,2}$ is also above and
    to the right of $x_{i,1}$, and the only vertex
    below and to the left of $x_{i-1,2}$ is $x_{i,1}$.  But
    this implies that for any induced matching $M$ of 
    $G \setminus N[x]$, $M \cup \{\{x,x_{i-1,2}\}\}$ is
    also an induced matching of $G$.  Hence ${\rm im}(G \setminus N[x])+1 \leq {\rm im}(G)$.  This now completes
    the proof.
\end{proof}

\begin{remark}\label{removeisolatedvertices}
    Let $G$ be a graph on $V = \{x_1,\ldots,x_n,z\}$
    and suppose that $z$ is an isolated vertex of $G$.
    Then $I(G)$ and $I(G^\circ)$ have the same
    generators, but $I(G)$ is an ideal
    of $R= k[x_1,\ldots,x_n,z]$ and $I(G^\circ)$
    is an ideal of $S =k[x_1,\ldots,x_n]$.  It follows
    that ${\rm reg}(R/I(G)) = {\rm reg}(S/I(G^\circ))$.  Consequently, Theorem
    \ref{thm:reg-downleft} also applies to
    $G(n,m,\vec{a},\vec{b})^\circ$, in the 
    appropriate polynomial ring.
\end{remark}
%Becuase of the above result, it would be nice
%to have an explicit formula in terms
%of $n,m, \vec{a}$ and $\vec{b}$.  In 
%some cases we can do this.
%We now turn our attention to the induced matching number $\mu(G)$ of the graph $G = G(n,m,\vec{a},\vec{b})$. 
In the special case of $G(m,n)$, we can compute
the regularity directly from $m$ and $n$.

\begin{theorem}\label{muG(n,m)}
If $G = G(m,n) = G(m,n,\vec{0},(n+1,\ldots,n+1))$,
then ${\rm reg}(R/I(G)) = \min \{m-1, n-1\}$.
\end{theorem}

\begin{proof}
Because $G(m,n) \cong G(n,m)$, without loss of generality, we can assume $m\leq n$. We first claim that $$E' = \big\{ \{x_{1,2},x_{2,1}\}, \{x_{2,3},x_{3,2}\}, \dots, \{x_{m-1,m},x_{m,m-1}\} \big\} \subset E(G(m,n))$$ is an induced matching of $G(m,n)$.  Let $V'$ be the set of vertices which are endpoints of the edges in $E'$ and let $H$ be the induced subgraph on $V'$.

The vertices of $H$ have the form $x_{i,i+1}$
for $i=1,\ldots,m-1$, and $x_{j+1,j}$ for
$j=1\ldots,m-1$.  A vertex of the
form $x_{i,i+1}$ cannot be adjacent to
any vertex of the form $x_{a,a+1}$ with $a\neq i$,
since these vertices are either above $x_{i,i+1}$
and to the left
if $a<i$ and to the right and below of $x_{i,i+1}$ if $a>i$.  Similarly, none of the vertices
of the form $x_{j+1,j}$ are adjacent to each other.
Finally, suppose that $x_{i,i+1}$ is adjacent
to $x_{j+1,j}$ with $i\neq  j$  If $i < j$, then we would
have $i+1 > j > i$.  But since $i$, $j$ and $i+1$
are integers, we cannot have $j$ strictly between
two consecutive integers.  Similarly, if $j <i$, then
$j < i <j+1$, which gives a similar problem.
So, the only induced edges on $H$ are those in $E'$,
and thus ${\rm im}(G) \geq m-1$.

%For any vertex $x_{s,t} \in H$, the only other vertex $x_{k,\ell} \in H$ that $x_{s,t}$ can share an edge with in $G(n,m)$ is $(t,s)$. This is because we need either $s>k$ and $t<\ell$ (so $(s,t)$ is down and to the left of $(k,\ell)$), or, $s<k$ and $t >\ell$, (so $(s,t)$ is up and to the right of $(k,\ell)$) for $(s,t)$ and $(k,\ell) $ to be adjacent. From this, we can conclude that the edges in the induced subgraph on $H$ are precisely those edges in $X$, so we conclude that $X$ is an induced matching of $G(n,m)$. Thus, we have that $\mu(G(n,m)) \geq n-1$.

Let us assume for the sake of a contradiction that there exists an induced matching $B$ of $G(m,n)$ such that $|B|>m-1$.  Given an
edge $e = \{x_{i,j},x_{k,\ell}\}$ of $G(m,n)$
with $i<k$, we denote $x_{i,j}$ as the \textit{top-right vertex} for the edge  $e$ (since each edge has a top-right vertex and bottom-left vertex). Observe that the $m$-th row cannot have a top-right vertex because that would imply there is a vertex in the $(m+1)$-th or lower row, which does not exist. As $B$ contains more than $m-1$ edges, by the Pigeonhole Principle there exists a row of $G(m,n)$ containing at least 
two top-right vertices of $B$. Suppose this is the $i$-th row. Let $x_{i,j_1}$ and $x_{i,j_2}$ be these top-right vertices with $j_1 < j_2$, and let $v$ denote the vertex such that $\{x_{i,j_1},v\} \in B$. Then when we take the induced subgraph on the vertices in $B$, since $v$ is also down and to the left of $x_{i,j_2}$, we have the edge 
$\{x_{i,j_2}, v\}$ in our induced subgraph, in addition to all edges in $B$. Therefore, both $\{x_{i,j_1}, v\}$ and $\{x_{i,j_2}, v\}$ must be edges in our induced subgraph, but they are not disjoint, contradicting our assumption that $B$ was an induced matching.   

Consequently, we conclude that 
${\rm im}(G(m,n))=m-1$ when $m \leq n$.
\end{proof}

\section{An application to toric ideals of graphs}

In this section we exhibit a connection
between down-left graphs and the toric ideals
of a special family of graphs. 
Given a finite simple  graph $G=(V,E)$ with 
vertex set $V=\{x_1,\dots,x_n\}$ and edge set $E=\{e_1,\dots,e_d\}$,
we define polynomial rings on the vertex and edge sets,
that is,
$k[V]=k[x_1,\dots,x_n]$ and $k[E]=k[e_1,\dots,e_d]$ where $k$ is any field. 
We define a monomial map $\pi:k[E]\rightarrow k[V]$ by $e_i\mapsto x_{i_1}x_{i_2}$ where $e_i=\{x_{i_1},x_{i_2}\} \in E$.  The kernel of 
$\pi:k[E]\rightarrow k[V]$ is called the 
\textit{toric ideal defined by $G$} and denoted $I_G$.
 %We define the image of this map $k[G]=k[E]/I_G$ to be the \textit{toric ring of $G$}.  
 It is well-known that the
closed even walks in $G$ generate $I_G$ \cite{OH1,V}.  In particular,
the ideal $I_G$ is a homogeneous prime ideal generated by binomials.

We are interested in combinatorial formulas or bounds on the regularity of the toric ideal of a graph to build upon
work of  \cite{ADS, BOVT, HH1}.  We will restrict to the class of chordal bipartite graphs in order to take advantage of the following result of Ohsugi and Hibi.

\begin{theorem}[{\cite{OH2}}]\label{quadraticGB}
If $G$ is a bipartite graph, then the following conditions are equivalent:
\begin{enumerate}
    \item $G$ is a chordal bipartite graph.
    \item  the toric ideal $I_G$ has a quadratic Gr\"obner basis.
\end{enumerate}
\end{theorem}

One consequence of this result is that for a chordal bipartite graph $G$, the initial ideal of $I_G$ in the order of Ohsugi-Hibi is a squarefree quadratic monomial ideal.  In other words, the initial ideal of $I_G$ in this order is the edge ideal of a different graph $H$ (we will construct this graph $H$ momentarily). In \cite{BOVT}, Biermann, O'Keefe, and Van Tuyl use the graph $H$ to give an upper bound on the regularity of $I_G$.  Our goal in this section will be to strengthen this upper bound to an exact calculation of the regularity in the case when the chordal bipartite graph is
also $(K_{3,3}\setminus e)$-free.  

The proof of \ref{quadraticGB} is constructive in that it explicitly produces an order and the elements of the quadratic Gr\"obner basis in that order.  The following two results are consequences of the proof which we extract here as a lemma for convenience.
Recall that if $G$ is a bipartite graph
with bipartition $\{x_1,\ldots,x_m\} \cup \{y_1,\ldots,y_n\}$, the {\it biadjacency matrix} of $G$
is an $m \times n$ matrix $M$ where $M_{i,j} =1$ if $\{x_i,y_j\}$ is
an edge, and $M_{i,j} = 0$ otherwise. 

\begin{lemma}[\cite{OH2}]\label{forbiddenSubmatrix}
Let $G$ be a chordal bipartite graph with
biadjacency matrix $M$. 
\begin{enumerate}
    \item The rows and columns of $M$ 
    can be rearranged so that they are simultaneously decreasing from left to right and from top to bottom in the reverse lexicographical order. 
   \item Assume that the rows and columns of $M$ are arranged so that they are decreasing from left to right and from top to bottom in the reverse lexicographical order.  Then there is no $2\times 2$ submatrix of $M$ of the form 
$$\begin{bmatrix}
        1&1\\
        1&0
    \end{bmatrix}.$$
\end{enumerate}
\end{lemma}

Let $M$ be a biadjacency matrix
of a chordal bipartite graph $G$.  
For the remainder of this paper, we will assume that we have relabelled the vertices
so that the rows and columns are ordered
as in Lemma \ref{forbiddenSubmatrix} (1).
Consequently, we can
relabel the edges of $G$ as $e_{i,j}$ if
$M_{i,j} = 1$.  We now make a new graph
$H$ from the relabelled matrix $M$ as follows.  The
vertex set of $H$ is
$V(H) = \{e_{i,j} ~|~ M_{i,j} = 1\}$ and
the edge set of $H$ is 
$$E(H) =\left\{ \{e_{a,d},e_{c,b}\}
~\left|~
\begin{tabular}{l}
$a< c, b <d$ and the submatrix of $M$ 
given by \\ rows $a$ and $c$ 
and columns $b$ and $d$
is $\begin{bmatrix}
        1&1\\
        1&1
    \end{bmatrix}$
    \end{tabular}
    \right\}\right..
    $$
Note that the graph $H$ is defined in terms
of the biadjacency matrix. 

As a consequence of Theorem \ref{quadraticGB}, the
initial ideal ${\rm in}_<(I_G)$ of $I_G$ is
related to the edge ideal of the graph $H$.  Note that
this result is implicit in \cite{OH2}, but 
made explicit in \cite{BOVT}.

\begin{theorem}\label{initialideal}
    Let $G$ be a chordal bipartite graph with
    biadjacency matrix $M$.  Let $H$ be the graph constructed from $M$
    as described above.  Then there 
    is a monomial order $>$ on $k[E]$ such that
    ${\rm in}_<(I_G) = I(H)$.
\end{theorem}

\begin{example}\label{runningexample}
We illustrate some of the above ideas with
the chordal bipartite graph given 
in Figure \ref{chordalbipartiteexample}.
This is the
same example as \cite[Example 4.6]{BOVT}.
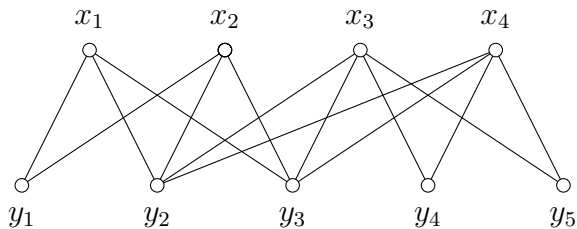
\begin{figure}[h!]
\begin{tikzpicture}[scale=.9]

\draw (1,2) -- (0,0);
\draw (1,2) -- (2,0);
\draw (1,2) -- (4,0);
\draw (3,2) -- (0,0); 
\draw (3,2) -- (2,0);
\draw (3,2) -- (4,0);
\draw (5,2) -- (2,0);
\draw (5,2) -- (4,0);
\draw (5,2) -- (6,0);
\draw (5,2) -- (8,0);
\draw (7,2) -- (2,0);
\draw (7,2) -- (4,0);
\draw (7,2) -- (6,0);
\draw (7,2) -- (8,0);

\fill[fill=white,draw=black] (1,2) circle (.1)
node[label=above:$x_1$] {};
\fill[fill=white,draw=black] (3,2)  circle (.1)circle (.1) node[label=above:$x_2$] {};
\fill[fill=white,draw=black] (5,2) circle (.1) node[label=above:$x_3$] {};
\fill[fill=white,draw=black] (7,2) circle (.1) node[label=above:$x_4$] {};
\fill[fill=white,draw=black](0,0)  circle (.1) node[label=below:$y_1$] {};
\fill[fill=white,draw=black](2,0)  circle (.1) node[label=below:$y_2$] {};
\fill[fill=white,draw=black](4,0)  circle (.1) node [label=below:$y_3$] {};
\fill[fill=white,draw=black](6,0)  circle (.1) node[label=below:$y_4$] {};
\fill[fill=white,draw=black](8,0)  circle (.1) node[label=below:$y_5$] {};

\end{tikzpicture}
\caption{A chordal bipartite graph}\label{chordalbipartiteexample}
\end{figure}
The vertex set has been partitioned as 
$\{x_1,x_2,x_3,x_4\} \cup \{y_1,y_2,y_3,y_4,y_5\}$.
Since the graph $G$ is a bipartite graph,
the biadjacency matrix $M$ is the $4 \times 5$ matrix given by
$$M = \begin{bmatrix}
1 & 1 & 1 & 0 & 0 \\
1 & 1 & 1 & 0 & 0 \\
0 & 1 & 1 & 1 & 1 \\
0 & 1 & 1 & 1 & 1
\end{bmatrix}.$$
By Lemma \ref{forbiddenSubmatrix}, we can relabel  the $x_i$'s and $y_j$'s,
so that the columns and rows of $M$ are 
arranged in reverse-lexicographical order, from
largest to smallest from top to bottom and
left to right.  In our example, the matrix
is rearranged\footnote[1]{There is a mistake in \cite[Example 4.6]{BOVT}; the rows in that
example are arranged in the 
correct reverse lexicographical order, but the columns were not rearranged in the correct order} as: 
$$M = \begin{bmatrix}
1 & 0 & 0 & 1 & 1 \\
1 & 0 & 0 & 1 & 1 \\
0 & 1 & 1 & 1 & 1 \\
0 & 1 & 1 & 1 & 1
\end{bmatrix}.$$
This relabelling 
amounts to swapping
the labels of $y_2$ and $y_4$, and those of $y_3$ and $y_5$.  Note that in accordance
to Lemma \ref{forbiddenSubmatrix} (2), the 
matrix has no submatrix of the form 
$$\begin{bmatrix}
        1&1\\
        1&0
    \end{bmatrix}.$$
We relabel our edges so  $e_{i,j}$ is connecting
$x_i$ with $y_j$.  Our new labelled
graph is given in Figure \ref{relabel}; we
have surpressed some of the edge labelling
to increase readability:
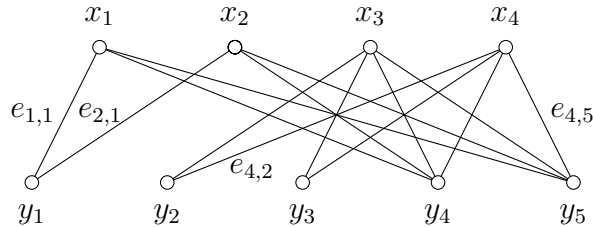
\begin{figure}[h!]
\begin{tikzpicture}[scale=.9]

\draw (1,2) -- (0,0) node[midway, left] {$e_{1,1}$};
\draw (1,2) -- (6,0);
\draw (1,2) -- (8,0);
\draw (3,2) -- (0,0) node[midway,left] {$e_{2,1}$}; 
\draw (3,2) -- (6,0);
\draw (3,2) -- (8,0);
\draw (5,2) -- (6,0);
\draw (5,2) -- (8,0);
\draw (5,2) -- (2,0);
\draw (5,2) -- (4,0);
\draw (7,2) -- (6,0);
\draw (7,2) -- (8,0)node[midway,right] {$e_{4,5}$};;
\draw (7,2) -- (2,0) node[near end,below]{$e_{4,2}$} ;
\draw (7,2) -- (4,0);

\fill[fill=white,draw=black] (1,2) circle (.1)
node[label=above:$x_1$] {};
\fill[fill=white,draw=black] (3,2)  circle (.1)circle (.1) node[label=above:$x_2$] {};
\fill[fill=white,draw=black] (5,2) circle (.1) node[label=above:$x_3$] {};
\fill[fill=white,draw=black] (7,2) circle (.1) node[label=above:$x_4$] {};
\fill[fill=white,draw=black](0,0)  circle (.1) node[label=below:$y_1$] {};
\fill[fill=white,draw=black](2,0)  circle (.1) node[label=below:$y_2$] {};
\fill[fill=white,draw=black](4,0)  circle (.1) node [label=below:$y_3$] {};
\fill[fill=white,draw=black](6,0)  circle (.1) node[label=below:$y_4$] {};
\fill[fill=white,draw=black](8,0)  circle (.1) node[label=below:$y_5$] {};

\end{tikzpicture}
\caption{The relabelled chordal bipartite graph with some of the edges labelled}\label{relabel}
\end{figure}

We now construct the graph $H$ from the matrix
$M$.  We place our vertices in a $4 \times 5$ grid,
where vertex $e_{i,j}$ is in row $i$ and column $j$ 
(similar to matrix notation).  We then join 
$e_{a,d}$ to $e_{c,b}$ if the corresponding
submatrix of rows $a$ and $c$ and columns $b$ and $d$
consists only of ones with
$a< c$ and $b<d$. The graph $H$ is given
in Figure \ref{graphH}.
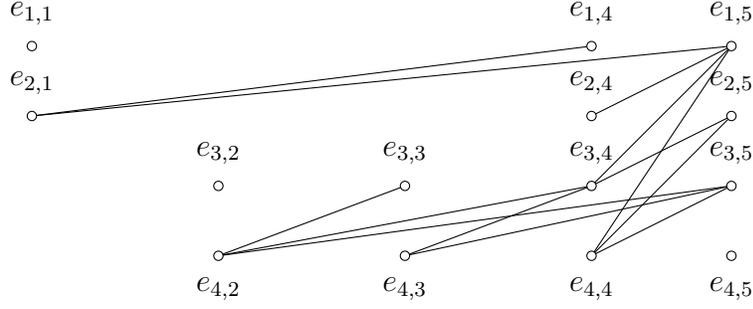
\begin{figure}[h!]
\begin{tikzpicture}[scale=0.31]

\draw (0,6) -- (24,9);
\draw (0,6) -- (30,9);

\draw (8,0) -- (16,3);
\draw (8,0) -- (24,3);
\draw (8,0) -- (30,3);

\draw (16,0) -- (24,3);
\draw (16,0) -- (30,3);

\draw (24,6) -- (30,9);
\draw (24,3) -- (30,9);
\draw (24,3) -- (30,6);
\draw (24,0) -- (30,9);
\draw (24,0) -- (30,3);
\draw (24,0) -- (30,6);

%\draw (0,0) -- (16,3);
%\draw (0,0) -- (24,3);
%\draw (0,0) -- (8,6);
%\draw (0,0) to[out=35, in=190] (16,6);
%\draw (8,0) to[out=35, in=190] (24,6);
%\draw (0,0) -- (24,6);

%\draw (8,0) -- (16,3);
%\draw  (16,0) -- (24,3);

%\draw (0,3) -- (8,6);
%\draw (8,3) -- (16,6);
%\draw (8,0) -- (24,3);
%\draw (8,0) -- (16,6);
%\draw (8,0) -- (24,6);
%\draw  (16,3) -- (24,6);
%\draw (16,0) -- (24,6);
%\draw (0,3) -- (8,6);
%%\draw (0,3) -- (16,6);
%\draw (0,3) -- (24,6);
%\draw (8,3) -- (24,6);
%\fill[fill=white,draw=black] (0,0) circle (.2);
%\fill[fill=white,draw=black] (0,3) circle (.2);
\fill[fill=white,draw=black] (0,6) circle (.2)
node[label=above:$e_{2,1}$] {};
\fill[fill=white,draw=black] (0,9) circle (.2)
node[label=above:$e_{1,1}$] {};

\fill[fill=white,draw=black] (8,0) circle (.2)
node[label=below:$e_{4,2}$] {};
\fill[fill=white,draw=black] (8,3) circle (.2)
node[label=above:$e_{3,2}$] {};
%\fill[fill=white,draw=black] (8,6) circle (.2);
%\fill[fill=white,draw=black] (8,9) circle (.2);

\fill[fill=white,draw=black] (16,0) circle (.2)
node[label=below:$e_{4,3}$] {};
\fill[fill=white,draw=black] (16,3) circle (.2)
node[label=above:$e_{3,3}$] {};
%\fill[fill=white,draw=black] (16,6) circle (.2);
%\fill[fill=white,draw=black] (16,9) circle (.2);

\fill[fill=white,draw=black] (24,0) circle (.2)
node[label=below:$e_{4,4}$] {};
\fill[fill=white,draw=black] (24,3) circle (.2)
node[label=above:$e_{3,4}$] {};
\fill[fill=white,draw=black] (24,6) circle (.2)
node[label=above:$e_{2,4}$] {};
\fill[fill=white,draw=black] (24,9) circle (.2)
node[label=above:$e_{1,4}$] {};

\fill[fill=white,draw=black] (30,0) circle (.2)
node[label=below:$e_{4,5}$] {};
\fill[fill=white,draw=black] (30,3) circle (.2)
node[label=above:$e_{3,5}$] {};
\fill[fill=white,draw=black] (30,6) circle (.2)
node[label=above:$e_{2,5}$] {};
\fill[fill=white,draw=black] (30,9) circle (.2)
node[label=above:$e_{1,5}$] {};;
\end{tikzpicture}
\caption{The graph $H$ constructed from $M$}\label{graphH}
\end{figure}
By Theorem \ref{initialideal}, there is 
a monomial ordering of $k[E]$ so that 
\begin{multline*}
{\rm in}_<(I_G) = I(H) =  \langle 
e_{2,\,1}e_{1,\,4},\; e_{2,\,1}e_{1,\,5},\; e_{2,\,4}e_{1,\,5},\; e_{3,\,4}e_{1,\,5},\; e_{4,\,4}e_{1,\,5},\; e_{3,\,4}e_{2,\,5},\; e_{4,\,4}e_{2,\,5}, \\
e_{4,\,2}e_{3,\,3}, \; e_{4,\,2}e_{3,\,4},\; e_{4,\,3}e_{3,\,4},\; e_{4,\,2}e_{3,\,5},\; e_{4,\,3}e_{3,\,5},\; e_{4,\,4}e_{3,\,5}
\rangle.
\end{multline*}
\end{example}

\begin{remark}
    The graph $H$ constructed from the 
    biadjacency matrix $M$ of the chordal 
    bipartite graph $G$ is similiar to
    a down-left graph.  In particular, if
    $e_{a,d}$ and $e_{c,b}$ are adjacent,
    then $e_{c,b}$ is down and to the left
    of $e_{a,d}$ if we place the vertices in 
    a grid.  Although similar, the graph
    $H$ may not  be a down-left graph.  
    In the above example, $e_{3,3}$ is down and
    left of $e_{2,4}$, but they are not adjacent
    since the corresponding submatrix made from
    rows $2$ and $3$ and columns $3$ and $4$ 
    contains a zero entry.  Moving forward,
    we want to find conditions on $M$ so that
    the resulting graph $H$ is also a down-left
    graph.
\end{remark}

The following is a technical lemma concerning the biadjacency matrix of a $(K_{3,3}\setminus e)$-free chordal bipartite graph which we will need in order to prove our main result. 

\begin{lemma} \label{lemma:techy}
Let $M$ be the $m\times n$ biadjacency matrix of a chordal bipartite $G$  that is also
$(K_{3,3}\setminus e)$-free.  Assume that the columns and rows of $M$
are ordered as in Lemma \ref{forbiddenSubmatrix} (1).
Let $M'$ be any $m' \times n'$ submatrix of $M$ with $m',n' \geq 2$ such that the entries of $M'$ are all 1's. Let $R$ be the set of indices of the rows of $M'$ as a submatrix of $M$ and let $C$ be the set of indices of the columns. 
Let $(s,t)$ be such that 
\begin{enumerate}
    \item $M_{s,t} = 1$,
    \item either $s \notin R$ or $t \notin C$, and
    \item there is some $r \in R$ and some $c \in C$ such that either $s<r$ and $t>c$, or $s>r$ 
    and $t<c$ with $M_{r,c} = M_{s,c} = M_{r,t} = 1$ and such that $(r,c) \neq (\max(R), \max(C))$.
\end{enumerate}
Let $M''$ be the submatrix of $M$ defined by the rows indexed by $R\cup\{s\}$ and columns indexed by $C\cup \{t\}$.  Then the entries of $M''$ are all 1.  
\end{lemma}

\begin{proof}

We need to prove that all the entries of $M''$ in row $s$ and column $t$ are non-zero.  It is sufficient to show that there is one element of row $s$ in $M''$ (other than $M_{s,t}$ and  $M_{s,c}$) which is non-zero or that there is one element of column $t$ (other than $M_{s,t}$ and $M_{r,t}$) which is non-zero.  To see this, assume without loss of generality that there exists some $r' \neq r$ with $r' \in R$ such that $M_{r', t} = 1$.  Then the submatrix of $M$ induced by the rows $s$, $r$, and $r'$ and columns $t$, $c$, and $j$ for any $j$ in $C$ contains eight entries equal to 1.  The final entry $M_{s,j}$ must be 1 as well since otherwise this submatrix is the biadjacency matrix of $K_{3,3} \setminus e$ and by assumption $G$ is $(K_{3,3}\setminus e)$-free.  This shows that all the entries in row $s$ in $M''$ are equal to 1.  A similar argument can then be applied to show all the entries of $M''$ in column $t$ are also equal to 1 and thus $M''$ is a matrix of 1's.  

Note that if either $s \in R$ or $t \in C$ then we are done since, for example, if $s \in R$ then we know automatically that every element in row $s$ of $M''$ is equal to 1.  Then either, there are only the two rows of $M''$ ($r$ and $s$) and we know that all the entries of $M''$ are 1, or $M''$ has at least three rows.  By assumption, we also know that $C$ contains at least two elements and thus $M''$ has at least three columns.  By the above argument, any remaining entries in column $t$ of $M''$ must be 1 to avoid a $3\times 3$ submatrix of $M$ with eight entries equal to 1 and one 0.  

Now suppose that $s \notin R$ and $t \notin C$.  Since taking the transpose of $M$ gives us the biadjacency matrix of an isomorphic graph, we may assume without loss of generality that $r<s$ and $t<c$.  

 \emph{Case 1:} Suppose there exists some $j \in C$ with $j >t$ and $j \neq c$.  This implies that
 $M_{r,j} = 1$. Then if $M_{s,j} = 0$ the submatrix on rows $r$ and $s$ and columns $t$ and $j$ contradicts Lemma \ref{forbiddenSubmatrix}.  Therefore we must have $M_{s,j} = 1$.  As discussed above, this is sufficient to show that $M''$ consists of all 1's.  

\emph{Case 2:} If there is no $j \in C$ with $j \neq c$ and $j >t$, then this means that columns $t$ and $c$ are the right-most two columns in $M''$.  Thus there exists some $j \in C$ with $j<c, t$.  

Now suppose that there exists an $i \in R$ such that $i>r$.  If $M_{i,t}=0$ then the submatrix induced by rows $r$ and $i$ and columns $j$ and $t$ contradicts Lemma \ref{forbiddenSubmatrix}.  Therefore we must have $M_{i,t} = 1$ and again by the above argument we have all the entries in $M''$ are 1.  

If on the other hand, there is no $i \in R$ such that $i >r$ then the rows $r$ and $s$ are the bottom-most rows in $M''$.  This together with the fact that columns $t$ and $c$ are the right-most in $M''$ means that $(r,c) = (\max{R}, \max{C})$ which contradicts our assumptions.
\end{proof}

\begin{example}\label{notk33efree}
The chordal bipartite graph of Example \ref{runningexample} is not $(K_{3,3}\setminus e)$-free since the induced graph on
$\{x_1,x_2,x_3,y_1,y_2,y_3\}$ in Figure \ref{chordalbipartiteexample} is a $K_{3,3}
\setminus e$.  
We show how the biadjacency
matrix does not satisfy the previous lemma
because the graph is not $(K_{3,3}\setminus e)$-free.
Note that the $2 \times 2$
submatrix of rows $3$ and $4$ and columns
$3$ and $4$ consists only of ones. Thus
$R = \{3,4\}$ and $C = \{3,4\}$.  We
have $M_{s,t} =M_{2,5} =1$ with $s= 2 \notin R$ and
$t= 5 \notin C$. Now
consider $r=3 > 2=s$ and $t=5>4=c$.  We have
$M_{3,4}=1 = M_{2,4} = M_{3,5}$.  So,
all three conditions of the previous lemma are
met, but the conclusion is false since
the submatrix consisting of rows $R \cup \{2\}$
and $C \cup \{5\}$ has a zero entry.
\end{example}
We are now prepared to prove our main theorem in this section, which says that the graph the $H$ constructed
from $G$ consists only of isolated vertices
or connected components that are  down-left
graphs of the form  $G(m,n)$ with the
isolated vertices removed.

\begin{theorem}\label{connectedComponents}
Let $G$ be a chordal bipartite graph, whose 
biadjacency matrix $M$ is ordered as in Lemma \ref{forbiddenSubmatrix},
and let $H$ be the corresponding 
graph constructed from $M$. If $G$ is also $(K_{3,3}\setminus e)$-free,
then each connected component of $H$ is either an isolated vertex or isomorphic to $G(m, n)^\circ$ for some integers $m,n \geq 2$. 
\end{theorem}

\begin{proof}
Assume for the sake of contradiction that there is a connected component $H'$ of $H$ that is not an isolated vertex and is not isomorphic to $G(m,n)^\circ$ for any $m,n \geq 2$.  Recall that the vertices of $H'$ correspond to 1's in $M$. Let $M'$ be the submatrix of $M$ that is given by the rows and columns which contain the vertices of $H'$.

Each edge in $H'$ corresponds to a $2\times 2$ submatrix of $M'$ whose entries are all 1.  Let $e$ be any edge in $H'$ and let $M''$ be a maximal submatrix of $M'$ which contains the $2 \times 2$ submatrix corresponding to $e$ and whose entries are all 1.      Let $H''$ be the induced subgraph of $H'$ on the vertices corresponding to the entries of $M''$.  Then if $M''$ is an $a\times b$ matrix, $H''$ is isomorphic to $G(a,b)^\circ \cup \{z_1,z_2\}$, where
$z_1$ and $z_2$ are the two isolated vertices
of $G(a,b)$.  Note that $a, b \geq 2$ since $H'$ contains at least one edge.
 Since $H'$ is a connected component of $H$ and $H'$ contains at least one edge of $G(a,b)^\circ$, $G(a,b)^\circ$ is a proper subgraph of $H'$.  Thus there must be some edge $\{v,v^*\}$ in $H'$ with $v$ a vertex of $G(a,b)^\circ$ and $v^*$ not a vertex of $H''$. 
 
 % Note that we may assume that $v$ is not the bottom-right vertex in $H''$ since if it was, then in order for $H'$ to be connected there must be a path in $H'$ from the edge $\{v, v^*\}$ to $H''$ to either $v$ or $v^*$ and another vertex in $H'$.  {\bf AVT: I'm a little unclear about this last sentence;  Do you mean ``We may assume that $v$ is not
 % the bottom-right vertex in $H''$; if it was,
 % then since $\{v,v^*\}$ and $G(a,b)^\circ$ are disjoint, there must be other edge(s)
 % in $H'$ that create a path between these two
 % components.  In particular, there must be an edge
 % $\{y,y^*\}$
 % in $H'$ with one endpoint in $G(a,b)^\circ$
 % and one in $H'$.  So, we can instead replace $\{v,v*\}$ with $\{y,y^*\}$.''} 

 By Lemma \ref{lemma:techy}, all of the entries of the submatrix of $M$ which consists of $M''$ plus the row and/or column containing $v^*$ are $1$.  This contradicts the fact that we chose $M''$ to be a maximal submatrix of all $1$'s. Therefore, every component of $H$ that is not an isolated vertex is isomorphic to $G(m,n)^\circ$ for some $m,n \geq 2$.
\end{proof}

Putting our results together, we can now give
an exact formula for the regularity of toric
ideals of chordal bipartite graphs that
are $(K_{3,3}\setminus e)$-free.

\begin{theorem}\label{thm:reg-chordalbipartite}
Let $G$ be a chordal
bipartite graph that is
$(K_{3,3,} \setminus e)$-free, and let 
$H$ be the graph constructed from
the biadjancey matrix $M$ of $G$. Since
$H$ has the form 
$H = G(m_1,n_1)^\circ \sqcup\cdots\sqcup G(m_s,n_s)^\circ \sqcup Y$ where $Y$ denotes
the set of isolated vertices, then
$${\rm reg}(k[E]/I_G) = 
\sum_{i=1}^s \min\{m_i-1,n_i-1\}.$$
\end{theorem}

\begin{proof}
Since the initial ideal ${\rm in}_<(I_G) = I(H)$ is a square-free quadratic monomial ideal, we know by Theorem \ref{regSqfreeInitialIdeal} that ${\rm reg}(k[E]/I_G) = {\rm reg}(k[E]/I(H))$.  The result then follows from Theorem \ref{connectedComponents}, Theorem \ref{muG(n,m)} and Remark \ref{removeisolatedvertices}, and Theorem \ref{regconnected}.
\end{proof}

\begin{example}
As noted in Example \ref{notk33efree},
our running example of Example \ref{runningexample}
is not a $(K_{3,3}\setminus e)$-free graph.
In fact,
any $3 \times 3$ submatrix of our biadjacency matrix
with exactly eight $1$'s and one $0$ corresponds
to an induced $K_{3,3}\setminus e$. 

Now remove the edges $\{x_2,y_4\}, \{x_3,y_3\}$ and
$\{x_4,y_2\}$ from this graph (again, using
the relabeled graph) to create the graph
in Figure \ref{edgesremoved}.
\begin{figure}[h!]
\begin{tikzpicture}[scale=.9]

\draw (1,2) -- (0,0);
\draw (1,2) -- (6,0);
\draw (1,2) -- (8,0);
\draw (3,2) -- (0,0); 
%\draw (3,2) -- (6,0);
\draw (3,2) -- (8,0);
\draw (5,2) -- (6,0);
\draw (5,2) -- (8,0);
\draw (5,2) -- (2,0);
%\draw (5,2) -- (4,0);
\draw (7,2) -- (6,0);
\draw (7,2) -- (8,0);
%\draw (7,2) -- (2,0);
\draw (7,2) -- (4,0);

\fill[fill=white,draw=black] (1,2) circle (.1)
node[label=above:$x_1$] {};
\fill[fill=white,draw=black] (3,2)  circle (.1)circle (.1) node[label=above:$x_2$] {};
\fill[fill=white,draw=black] (5,2) circle (.1) node[label=above:$x_3$] {};
\fill[fill=white,draw=black] (7,2) circle (.1) node[label=above:$x_4$] {};
\fill[fill=white,draw=black](0,0)  circle (.1) node[label=below:$y_1$] {};
\fill[fill=white,draw=black](2,0)  circle (.1) node[label=below:$y_2$] {};
\fill[fill=white,draw=black](4,0)  circle (.1) node [label=below:$y_3$] {};
\fill[fill=white,draw=black](6,0)  circle (.1) node[label=below:$y_4$] {};
\fill[fill=white,draw=black](8,0)  circle (.1) node[label=below:$y_5$] {};

\end{tikzpicture}
\caption{The relabelled chordal bipartite graph with some of the edges labelled}\label{edgesremoved}
\end{figure}
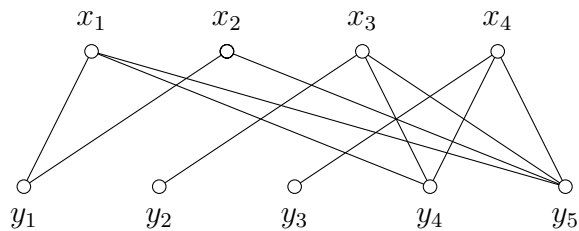
The new biadjacency matrix arranged in reverse-lexicographical order from
largest to smallest from top to bottom and
left to right is given by
$$M' = \begin{bmatrix}
1 & 0 & 0 & 0 & 1 \\
1 & 0 & 0 & 1 & 1 \\
0 & 1 & 0 & 1 & 1 \\
0 & 0 & 1 & 1 & 1
\end{bmatrix}.$$
From the matrix, we see that the new graph
is still a chordal bipartite graph and has no
induced $(K_{3,3}\setminus e)$.  The graph
$H'$ made from the matrix $M'$ is given
in Figure \ref{graphHprime}. 
\begin{figure}[h!]
\begin{tikzpicture}[scale=0.27]

%\draw (0,6) -- (24,9);
\draw (0,6) -- (30,9);

%\draw (8,0) -- (16,3);
%\draw (8,0) -- (24,3);
%\draw (8,0) -- (30,3);

%\draw (16,0) -- (24,3);
%\draw (16,0) -- (30,3);

%\draw (24,6) -- (30,9);
%\draw (24,3) -- (30,9);
\draw (24,3) -- (30,6);
%\draw (24,0) -- (30,9);
\draw (24,0) -- (30,3);
\draw (24,0) -- (30,6);

%\draw (0,0) -- (16,3);
%\draw (0,0) -- (24,3);
%\draw (0,0) -- (8,6);
%\draw (0,0) to[out=35, in=190] (16,6);
%\draw (8,0) to[out=35, in=190] (24,6);
%\draw (0,0) -- (24,6);

%\draw (8,0) -- (16,3);
%\draw  (16,0) -- (24,3);

%\draw (0,3) -- (8,6);
%\draw (8,3) -- (16,6);
%\draw (8,0) -- (24,3);
%\draw (8,0) -- (16,6);
%\draw (8,0) -- (24,6);
%\draw  (16,3) -- (24,6);
%\draw (16,0) -- (24,6);
%\draw (0,3) -- (8,6);
%%\draw (0,3) -- (16,6);
%\draw (0,3) -- (24,6);
%\draw (8,3) -- (24,6);
%\fill[fill=white,draw=black] (0,0) circle (.2);
%\fill[fill=white,draw=black] (0,3) circle (.2);
\fill[fill=white,draw=black] (0,6) circle (.2)
node[label=above:$e_{2,1}$] {};
\fill[fill=white,draw=black] (0,9) circle (.2)
node[label=above:$e_{1,1}$] {};

%\fill[fill=white,draw=black] (8,0) circle (.2)
%node[label=below:$e_{4,2}$] {};
\fill[fill=white,draw=black] (8,3) circle (.2)
node[label=above:$e_{3,2}$] {};
%\fill[fill=white,draw=black] (8,6) circle (.2);
%\fill[fill=white,draw=black] (8,9) circle (.2);

\fill[fill=white,draw=black] (16,0) circle (.2)
node[label=below:$e_{4,3}$] {};
%\fill[fill=white,draw=black] (16,3) circle (.2)
%node[label=above:$e_{3,3}$] {};
%\fill[fill=white,draw=black] (16,6) circle (.2);
%\fill[fill=white,draw=black] (16,9) circle (.2);

\fill[fill=white,draw=black] (24,0) circle (.2)
node[label=below:$e_{4,4}$] {};
\fill[fill=white,draw=black] (24,3) circle (.2)
node[label=above:$e_{3,4}$] {};
\fill[fill=white,draw=black] (24,6) circle (.2)
node[label=above:$e_{2,4}$] {};
%\fill[fill=white,draw=black] (24,9) circle (.2)
%node[label=above:$e_{1,4}$] {};

\fill[fill=white,draw=black] (30,0) circle (.2)
node[label=below:$e_{4,5}$] {};
\fill[fill=white,draw=black] (30,3) circle (.2)
node[label=above:$e_{3,5}$] {};
\fill[fill=white,draw=black] (30,6) circle (.2)
node[label=above:$e_{2,5}$] {};
\fill[fill=white,draw=black] (30,9) circle (.2)
node[label=above:$e_{1,5}$] {};;
\end{tikzpicture}
\caption{The new graph $H'$ constructed from $M'$.  This graph is a $G(2,2)^\circ \sqcup G(3,2)^\circ$ plus five isolated vertices}\label{graphHprime}
\end{figure}
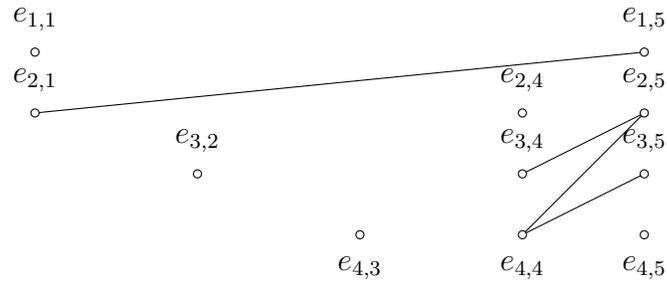
 Observe
that $H'$ is the disjoint union of a 
$G(2,2)^\circ$ and $G(3,2)^\circ$ and some
isolated vertices. Thus
$${\rm reg}(R/I_G) = \min\{2-1,2-1\} + \min\{3-1,2-1\} = 2.$$
\end{example}

\end{document}